\documentclass[a4paper,12pt]{article}

\usepackage{amsfonts}
\usepackage{amscd,color}
\usepackage{amsmath,amsfonts,amssymb,amscd}
\usepackage{indentfirst,graphicx,epsfig}
\usepackage{graphicx}
\input{epsf}
\usepackage{graphicx}
\usepackage{epstopdf}
\usepackage{caption}

\newtheorem{thm}{Theorem}[section]

\newtheorem{lem}[thm]{Lemma}

\newenvironment {proof} {\noindent{\em Proof.}}{\hspace*{\fill}$\Box$\par\vspace{4mm}}
\newcommand{\ml}{l\kern-0.55mm\char39\kern-0.3mm}

\title{\textbf{Erd\H{o}s-Gallai-type results for the rainbow disconnection number of graphs\footnote{Supported by NSFC No.11871034, 11531011 and NSFQH No.2017-ZJ-790.}}}
\author{{\small Xuqing Bai$^1$, Xueliang Li$^{1,2}$ } \\
{\small  $^1$Center for Combinatorics and LPMC}\\
{\small Nankai University, Tianjin 300071, China}\\
{\small Email: baixuqing0@163.com, lxl@nankai.edu.cn}\\
{\small $^2$School of Mathematics and Statistics, Qinghai Normal University}\\
{\small Xining, Qinghai 810008, China}\\
}
\date{}
\begin{document}
\maketitle
\begin{abstract}
Let $G$ be a nontrivial connected and edge-colored graph.
An edge-cut $R$ of $G$ is called a rainbow cut if no
two edges of it are colored with a same color. An edge-colored
graph $G$ is called rainbow disconnected if for every two
distinct vertices $u$ and $v$ of $G$, there exists a $u-v$ rainbow cut
separating them.
For a connected graph $G$, the rainbow disconnection
number of $G$, denoted by $rd(G)$, is defined as the
smallest number of colors that are needed in order to
make $G$ rainbow disconnected. In this paper, we will
study the Erd\H{o}s-Gallai-type results for
$rd(G)$, and completely solve them.


\end{abstract}

\section{Introduction}

All graphs considered in this paper are simple, finite
and undirected. Let $G=(V(G),E(G))$ be a nontrivial
connected graph with vertex set $V(G)$ and edge
set $E(G)$. For $v\in V(G)$, let $N_G(v)$ and $N_G[v]$
denote the \emph{open neighbour} of $v$ and the
\emph{closed neighbour} of $v$ in $G$, respectively.
For any notation or terminology not defined here, we
follow those used in \cite{BM}.

Throughout this paper, we use
$K_n$ to denote a complete graph of order $n$.
A \emph{$k$-factor} of $G$ is a $k$-regular spanning
subgraph of $G$, and $G$ is \emph{$k$-factorable} if
there are edge-disjoint $k$-factors $H_1$, $H_2$,
\ldots, $H_n$ such that $G=H_1\cup H_2\cup \ldots \cup H_n$.
A subset $M$ of $E$ is called a \emph{matching} of
$G$ if any two edges of $M$ do not share a common vertex of $G$.

Let $G$ be a graph with an \emph{edge-coloring} $c$:
$E(G)\rightarrow [k] = \{1,2,...,k\}$, $k \in \mathbb{N}$,
where adjacent edges may be colored the same. When
adjacent edges of $G$ receive different colors under $c$,
the edge-coloring $c$ is called \emph{proper}.
The \emph{chromatic index} of $G$, denoted by $\chi'(G)$,
is the minimum number of colors needed in a proper
coloring of $G$. By a famous theorem of Vizing \cite{V},
one has
$$\Delta(G) \leq \chi'(G) \leq \Delta(G)+1$$
for every nonempty graph $G$.

A path is called \emph{rainbow} if no two edges of it are
colored the same. An edge-colored graph $G$ is called
\emph{rainbow~connected} if every two distinct vertices of $G$ are
connected by a rainbow path in $G$. An edge-coloring under
which $G$ is rainbow connected is called a
\emph{rainbow~connection~coloring}. Clearly, if a
graph is rainbow connected, it must be connected.
For a connected graph $G$, the
\emph{rainbow connection number} of $G$, denoted by
$rc(G)$, is the smallest number of colors that are
needed to make $G$ rainbow connected. Rainbow
connection was introduced by Chartrand et al.
\cite{CJMZ} in $2008$. For more details on rainbow
connection, see the book \cite{LS1} and the survey
papers \cite{LSS, LS2}.

In this paper, we investigate a new concept that
is somewhat reverse to rainbow connection. This
concept of rainbow disconnection of graphs was
introduced by Chartrand et al. \cite{CDHHZ} very
recently in $2018$.

An \emph{edge-cut} of a connected graph $G$ is a set $S$ of
edges such that $G-S$ is disconnected. The minimum
number of edges in an edge-cut is defined as the
\emph{edge-connectivity} $\lambda(G)$ of $G$. We have the
well-known inequality $\lambda(G)\leq \delta(G)$.
For two vertices $u$ and $v$, let $\lambda(u,v)$
denote the minimum number of edges in an edge-cut
$S$ such that $u$ and $v$ lie in different components
of $G-S$.
The so-called \emph{upper edge-connectivity}
$\lambda^+(G)$ of $G$ is defined by
$$\lambda^+(G) = max\{\lambda(u,v): u, v\in V(G)\}.$$
$\lambda^+(G)$ is the maximum local edge-connectivity of $G$,
while $\lambda(G)$ is the  minimum global
edge-connectivity of $G$.

An edge-cut $R$ of an edge-colored connected graph $G$ is called a {\it rainbow cut} if no
two edges in $R$ are colored the same. A rainbow cut
$R$ is said to separate two vertices $u$ and $v$ if
$u$ and $v$ belong to different components of $G-R$.
Such rainbow cut is called a $u-v$ rainbow cut. An
edge-colored graph $G$ is called \emph{rainbow disconnected}
if for every two distinct vertices $u$ and $v$ of $G$, there
exists a $u-v$ rainbow cut in $G$. In this case,
the edge-coloring $c$ is called a
\emph{rainbow disconnection coloring} of $G$.
Similarly, we define the \emph{rainbow disconnection number}
of a connected graph $G$, denoted by $rd(G)$, as the smallest number of
colors that are needed in order to make $G$ rainbow
disconnected. A rainbow disconnection coloring with
$rd(G)$ colors is called an $rd$-\emph{coloring} of $G$.

The Erd\H{o}s-Gallai-type problem is an interesting
problem in extremal graph theory, which was studied
in \cite{LLSZ, LLS, L} for rainbow connection number
$rc(G)$; in \cite{HLW} for proper connection number
$pc(G)$; in \cite{CLW} for monochromatic connection
number $mc(G)$, and many other parameter of graphs in
literature. We will study the Erd\H{o}s-Gallai-type
results for the rainbow disconnection number $rd(G)$
in this paper.

\section{ Preliminary results}

For given integers $k$ and $n$ with
$1 \leq k \leq n-1$, the authors in \cite{CDHHZ} determined the
minimum size of a connected graph $G$ of order $n$
with $rd(G) = k$. So, this brings up the question of
determining the maximum size of a connected graph
$G$ of order $n$ with $rd(G) = k$.
The authors of \cite{CDHHZ} conjectured and we determined in \cite{BCL}
the maximum size of a connected graph $G$ of order $n$
with $rd(G) = k$, for odd integer $n$. But for even integer $n$,
it was left without solution. Now we consider the question of
determining the maximum size of a connected graph
$G$ of even order $n$ with $rd(G) = k$
and we get the following result.

\begin{thm}\label{even}
Let $k$ and $n$ be integers with $1\leq k\leq n-1$
and $n$ be even. Then the maximum size of a connected
graph $G$ of order $n$ with $rd(G)=k$
is $\lfloor\frac{(k+1)(n-1)}{2}\rfloor$.
\end{thm}
Before we give the proof of Theorem \ref{even},
some useful lemmas are stated as follows.
\begin{lem}{\upshape\cite{CDHHZ}}\label{p1}
If $G$ is a nontrivial connected graph, then
$$\lambda(G) \leq \lambda^+(G)\leq rd(G)\leq
\chi'(G) \leq \Delta(G)+1 .$$
\end{lem}
\begin{lem}{\upshape\cite{CDHHZ}}\label{p2}
Let $G$ be a nontrivial connected graph.
Then $rd(G) = 1$ if and only if $G$ is a tree.
\end{lem}

\begin{lem}{\upshape\cite{CDHHZ}}\label{com}
For each integer $n\geq 4$, $rd(K_n)=n-1$.
\end{lem}
\noindent\textbf{Remark 1:} For any integer
$n\geq 2$, $rd(K_n)=n-1$ since it is easy to
verify that $rd(K_2)=1$ and $rd(K_3)=2$.

\begin{lem}{\upshape\cite{M}}\label{key}
Let $G$ be a graph of order $n$ $(n\geq k+2 \geq 3)$. If
$e(G)>\frac{k+1}{2}(n-1)-\frac{1}{2}\sigma_k(G),$ where
$\sigma_k(G)=\sum\limits_{\mbox{\tiny $\begin{array}{c}
             x\in V(G) \\
             d(x)\le k \end{array}$}}(k-d(x))$, then
$\lambda^+(G)\geq k+1.$
\end{lem}

\begin{lem}\label{2}
If $n$ is even, then there exists a $k$-regular graph
$G$ of order $n$, where $1\leq k\leq n-2$, that satisfies
both of the following conditions:

(i) $G$ is $1$-factorable.

(ii) there exists a vertex $u$ of $G$ such that $G[N(u)]$
can add $\lfloor\frac{k}{2}\rfloor$ matching edges.
\end{lem}

\begin{proof}
Let $G_1=K_{2n}$. Since $K_{2n}$ is $1$-factorable,
$K_{2n}$ has $2n-1$ edge-disjoint $1$-factors.
First, we remove a $1$-factor from $K_{2n}$, namely
$n$ matching edges. Let $e_1$, $e_2$,\ldots, $e_n$ be
$n$ removing matching edges of $G_1$ and let $v_{i,1}$ and
$v_{i,2}$ be the two end vertices of $e_{i}$. Let
$u=v_{n,1}$ be a vertex of $G_1$. Then the remaining
graph $G_2$ is a $(2n-2)$-regular graph with
$N_{G_2}(u)=\{v_{i,1},v_{i,2}|1\leq i\leq n-1\}$,
and $\{e_i|1\leq i\leq n-1\}$ are matching edges
which can be added to $G_2[N(u)]$ and the number
of matching edges is $\lfloor\frac{2n-2}{2}\rfloor$.

Second, we remove the $1$-factor from $G_2$ which
contains the edge $uv_{n-1,1}$, and denote remaining
graph by $G_3$. Obviously, $G_3$ is a $(2n-3)$-regular
graph and $N_{G_3}(u)$ is $N_{G_2}(u)\setminus v_{n-1,1}$,
and $\{e_i|1\leq i\leq n-2\}$ are matching edges
which can be added to $G_3[N(u)]$ and the number of
matching edges is $\lfloor\frac{2n-3}{2}\rfloor$.

Third, we remove the $1$-factor from $G_3$ which
contains the edge $uv_{n-1,2}$, and denote remaining
graph by $G_4$. Obviously, $G_4$ is a $(2n-4)$-regular
graph and $N_{G_4}(u)$ is $N_{G_3}(u)\setminus v_{n-1,2}$,
and $\{e_i|1\leq i\leq n-2\}$ are matching edges
which can be added to $G_4[N(u)]$ and the number
of matching edges is $\lfloor\frac{2n-4}{2}\rfloor$.

For $G_j$ ($j\geq 2$), if $j$ is even, then we remove the
$1$-factor from $G_j$ which contains the edge
$uv_{n-\lfloor\frac{j}{2}\rfloor,1}$; if $j$ is odd,
then we remove the $1$-factor from $G_j$ which contains
the edge $uv_{n-\lfloor\frac{j}{2}\rfloor,2}$.
Obviously, the remaining graph $G_{j+1}$ is a
$(2n-j-1)$-regular graph and $N_{G_{j+1}}(u)$ is
$N_{G_j}(u)\setminus v_{n-\lfloor\frac{j}{2}\rfloor,i}$
($i=1$ or $2$), and
$\{e_i|1\leq i\leq n-\lfloor\frac{j}{2}\rfloor-1\}$
are matching edges which can be added to $G_{j+1}[N(u)]$
and the number of matching edges is
$\lfloor\frac{2n-j-1}{2}\rfloor$.

Repeating the above process, we can get a $k$-regular graph $G_{2n-k}$
of order $n$ which is $1$-factorable. Furthermore,
$\{e_i|1\leq i\leq n-\lfloor\frac{2n-k}{2}\rfloor\}$
are matching edges which can be added to $G_{2n-k}[N(u)]$
and the number of matching edges is $\lfloor\frac k{2}\rfloor$.
\end{proof}

Now we are ready to give a proof to Theorem 2.1.

\noindent {\bf Proof~ of~ Theorem ~\ref{even}:}
It is easy to see that the graphs $G$ of maximum size with order $n$
and $rd(G)=k$ is not more than $\frac{(k+1)(n-1)}{2}$.
Otherwise, $rd(G)\geq k+1$ by Lemmas \ref{p1} and
\ref{key}. Now we show that the graphs $G$ of maximum size
with even order $n$ and $rd(G)=k$ is
$\lfloor\frac{(k+1)(n-1)}{2}\rfloor$ for $1\leq k\leq n-1$.
For $k=n-1$, let $G=K_n$. Note that $|E(G)|=
\lfloor\frac{(k+1)(n-1)}{2}\rfloor$ and $rd(G)=n-1$
by Remark $1$.
For $k=1$, let $G$ be a tree. Note that
$|E(G)|=\lfloor\frac{(k+1)(n-1)}{2}\rfloor$ and
$rd(G)=1$ by Lemma \ref{p2}.
Now we construct a graph $G$ for $2\leq k\leq n-2$
as follows.
Let $H_{k-1}$ be a $(k-1)$-regular graph of order $n$.
For $k\geq 2$, $H_{k-1}$ can be selected so that it
is $1$-factorable and there exists a vertex $u$ of
$H_{k-1}$ such that $\lfloor\frac{k-1}{2}\rfloor$
matching edges can be added to $N_{H_{k-1}}(u)$ by
Lemma \ref{2}.
Let $G$ be a graph by adding $\lfloor\frac{k-1}{2}\rfloor$
matching edges to $N_{H_{k-1}}(u)$ and adding $n-k$ edges
of $\{uw|w\in V(H_{k-1})\setminus N_{H_{k-1}}[u]\}$ in
$H_{k-1}$.
Thus, $G$ is a graph of order $n$ with $|E(G)|=\frac{(k-1)n}{2}+\lfloor\frac{k-1}{2}\rfloor+n-k
=\lfloor\frac{(k+1)(n-1)}{2}\rfloor$.
Since $\chi'(H_{k-1})=k-1$, we obtain a proper
edge-coloring $c_0$ of $H_{k-1}$ using colors from
$[k-1]$. We may extend $c_0$ to an edge-coloring $c$
of $G$ by assigning a fresh color $k$ to all newly
added edges in $H_{k-1}$.
Note that the set $E_x$ of edges incident with $x$
in $G$ is a rainbow set for each vertex
$x\in V(G)\setminus u$. Let $p$ and $q$ be two
vertices of $G$. Then at least one of $p$ and $q$
is not $u$, say $p\neq u$. Since $E_p$ is a $p-q$
rainbow cut, $c$ is a rainbow disconnection coloring
of $G$ using at most $k$ colors. Therefore, $rd(G)\leq k$.
On the other hand, $E(G)=\lfloor\frac{(k+1)(n-1)}{2}\rfloor
> \frac{k(n-1)}{2}$ since $n\geq 3$, it follows
from Lemmas \ref{p1} and \ref{key} that $rd(G)\geq k$.
$~~~~~~~~~~~~~~~~~~~~~~~~~~~~~~~~~~~~~~~~~~~~~~~~~~
~~~~~~~~~~~~~~~~~~~~~~~~~~~~~~~~~~~~~~~~~~~~~~~~~~~\Box$

\section{Erd\H{o}s-Gallai-type results for $rd(G)$ }

Now we consider the following two kinds of
Erd\H{o}s-Gallai-type problems for $rd(G)$.

\textbf{Problem A}. Given two positive integers
$n$ and $k$ with $1\leq k \leq n-1$, compute the
maximum integer $g(n, k)$ such that for any graph
$G$ of order $n$, if $|E(G)|\leq g(n, k)$, then
$rd(G)\leq k$.

\textbf{Problem B}. Given two positive integers
$n$ and $k$ with $1\leq k \leq n-1$, compute the
minimum integer $f(n, k)$ such that for any graph
$G$ of order $n$, if $|E(G)|\geq f(n, k)$ then
$rd(G)\geq k$.

It is worth mentioning that the two parameters
$f(n, k)$ and $g(n, k)$ are equivalent to
another two parameters. Let
$t(n,k)$=$min\{|E(G)|:|V(G)|= n, rd(G)\geq k\}$
and
$s(n, k)$=$max\{|E(G)|:|V(G)|=n, rd(G)\leq k\}$.
It is easy to see that $g(n,k) = t(n,k+1)-1$ and
$f(n,k) = s(n,k-1)+1$.

We first state two lemmas, which will be used
to determine the values of $f(n,k)$ and $t(n,k)$.

\begin{lem}{\upshape\cite{CDHHZ}}\label{min}
For integers $k$ and $n$ with $1\leq k \leq n-1$,
the minimum size of a connected graph of order
$n$ with $rd(G)=k$ is $n+k-2$.
\end{lem}

Note that the following result from \cite{BCL} is also true
for $n=1,3$. So we can state it as follows, without $n\geq 5$.

\begin{lem}{\upshape\cite{BCL}}\label{odd}
Let $k$ and $n$ be integers with $1 \leq k \leq n-1$
and $n$ be odd.
Then the maximum size of a connected graph
$G$ of order $n$ with $rd(G)=k$
is $\frac{(k+1)(n-1)}{2}$.
\end{lem}

Using Lemma \ref{min}, we first solve Problem $A$.
\begin{thm}
$g(n,k)=n+k-2$ for $1\leq k\leq n-1$.
\end{thm}

\begin{proof}
It follows from Lemma \ref{min} that $t(n,k)=n+k-2$
for $1\leq k\leq n-1$.
Thus, $g(n,k)=t(n,k+1)-1=n+k-2$.
\end{proof}
Now we come to the solution for Problem $B$,
we get the following result.
\begin{thm}
$f(n,k)=\lfloor\frac{k(n-1)}{2}\rfloor+1$ for
$1\leq k\leq n-1$.
\end{thm}

\begin{proof}
If $n$ is odd, then $s(n,k)=\frac{(k+1)(n-1)}{2}$
for $1\leq k\leq n-1$ from Lemma \ref{odd}.
Thus, $f(n,k)=s(n,k-1)+1=\frac{k(n-1)}{2}+1
=\lfloor\frac{k(n-1)}{2}\rfloor+1$
for $1\leq k\leq n-1$.

If $n$ is even, then $s(n,k)=
\lfloor\frac{(k+1)(n-1)}{2}\rfloor$ for
$1\leq k\leq n-1$ from Theorem \ref{even}.
Thus, $f(n,k)=s(n,k-1)+1=
\lfloor\frac{k(n-1)}{2}\rfloor+1$
for $1\leq k\leq n-1$ for $n$ is even.
\end{proof}

\end{document}